\documentclass{article}

\usepackage{graphicx}
\usepackage{amsmath}
\usepackage{amssymb}
\usepackage{color}
\usepackage{hyperref}
\usepackage{epsfig}
\allowdisplaybreaks

\newtheorem{theorem}{Theorem}

\newtheorem{coro}{Corollary}

\newtheorem{remark}{Remark}

\newcommand{\NN}{{\mathbb N}}

\newenvironment{proof}{\begin{trivlist}
   \item[\hskip\labelsep{\it Proof.}]}{$\hfill\Box$\end{trivlist}}

\newcommand{\notiz}[1]{}

\title{The Champernowne constant is not Poissonian}
\author{\'Isabel Pirsic \footnote{The author is supported by the Austrian Science Fund (FWF), Project P27351-N26”.}\ , Wolfgang Stockinger \footnote{The author is supported by the Austrian Science Fund (FWF), Project F5507-N26, which is a part of the Special Research Program “Quasi-Monte Carlo Methods: Theory and Applications”.}}
\date{}
\begin{document}
\maketitle 
\begin{abstract}
We say that a sequence $(x_n)_{n \in \NN}$ in $[0,1)$ has Poissonian pair correlations if 
\begin{equation*}
\lim_{N \to \infty} \frac{1}{N} \# \left\lbrace 1  \leq l \neq m \leq N: \| x_l - x_m \| \leq \frac{s}{N} \right\rbrace = 2s
\end{equation*} 
for every $s \geq 0$. In this note we study the pair correlation statistics for the sequence of shifts of $\alpha$, $x_n = \lbrace 2^n \alpha \rbrace, \ n=1, 2, 3, \ldots$, where we choose $\alpha$ as the Champernowne constant in base $2$. Throughout this article $\lbrace \cdot \rbrace$ denotes the fractional part of a real number. It is well known that $(x_n)_{n \in \NN}$ has Poissonian pair correlations for almost all normal numbers $\alpha$ (in the sense of Lebesgue), but we will show that it does not have this property for all normal numbers $\alpha$, as it fails to be Poissonian for the Champernowne constant.    
\end{abstract} 
\section{Introduction and main result}
The concept of Poissonian pair correlations has its origin in quantum mechanics, where the spacings of energy levels of integrable systems were studied. See for example \cite{not9} and the references therein for detailed information on that topic. Rudnik and Sarnak first studied this concept from a purely mathematical point of view and over the years the topic has attracted wide attention, see e.g., \cite{ not5, not10, not1, not2, not3}. Recently, Aistleitner, Larcher and Lewko (see \cite{not6}) could give a strong link between the concept of Poissonian pair correlations and the additive energy of a finite set of integers, a notion that plays an important role in many mathematical fields, e.g., in additive combinatorics. Roughly speaking, they proved that if the first $N$ elements of an increasing sequence of distinct integers $(a_n)_{n \in \NN}$, have an arbitrarily small energy saving, then $(\lbrace a_n \alpha \rbrace)_{n \in \NN}$ has Poissonian pair correlations for almost all $\alpha$. This result implies the metrical Poissonian pair correlation property for lacunary sequences as well. In this paper the authors also raised the question if an increasing sequence of distinct integers with maximal order of additive energy can have Poissonian pair correlations for almost all $\alpha$. Jean Bourgain could show that the answer to this question is negative, see the appendix of \cite{not6} for details and a second problem which was also solved by Bourgain. Recently, the results of Bourgain have been further extended, see \cite{not9, not14, not15, not16}.   \\ 

Let $\| \cdot \|$ denote the distance to the nearest integer. A sequence $(x_n)_{n \in \NN}$ of real numbers in $[0,1)$ has Poissonian pair correlations if 
\begin{equation}\label{eq:pc1}
\lim_{N \to \infty} \frac{1}{N} \# \left\lbrace 1  \leq l \neq m \leq N: \| x_l - x_m \| \leq \frac{s}{N} \right\rbrace = 2s
\end{equation} 
for every $s \geq 0$. Due to a result by Grepstad and Larcher \cite{not8} (see also \cite{not7, not12}), we know that a sequence which satisfies property (\ref{eq:pc1}), is also uniformly distributed in $[0,1)$, i.e., it satisfies 
\begin{equation*}
\lim_{N \to \infty} \frac{1}{N} \# \lbrace 1 \leq n \leq N: x_n \in [a,b) \rbrace = b-a 
\end{equation*}
for all $0 \leq a < b \leq 1$. Note that the other direction is not necessarily correct. For instance the Kronecker sequence $(\lbrace n\alpha \rbrace)_{n \in \NN}$, does not have this property for any real $\alpha$; a fact that can be argued by a continued fractions argument or by the main theorem in \cite{not19} in combination with the famous Three Gap Theorem, see \cite{not4}. Poissonian pair correlation is a typical property of a sequence. Random sequences, i.e., almost all sequences, have Poissonian pair correlations. Nevertheless, it seems to be extremely difficult to give explicit examples of sequences with Poissonian pair correlations. We note that $(\lbrace \sqrt{n} \rbrace)_{n \in \NN}$ has Poissonian pair correlations, \cite{not17} (see \cite{not18} for another explicit construction). Apart from that -- to our best knowledge -- no other explicit examples are known. Especially, until now we do not know any single explicit construction of a real number $\alpha$ such that the sequence of the form $(\lbrace a_n \alpha \rbrace)_{n \in \NN}$ has Poissonian pair correlations.\\
 
We recall that the sequence $(\lbrace 2^n \alpha \rbrace)_{n \in \NN}$ has Poissonian pair correlations for almost all $\alpha$. In this note, we study the distribution of the pair correlations of the sequence $(\lbrace 2^n \alpha \rbrace)_{n \in \NN}$, where $\alpha$ is the Champernowne constant in base $2$, i.e., $\alpha = 0.1101110010111011 \ldots_2$. It is a well known fact that the Champernowne constant in base $2$ is normal to base 2. Moreover we know that the sequence  $(\lbrace 2^n \alpha \rbrace)_{n \in \NN}$ is uniformly distributed modulo $1$ if and only if $\alpha$ is normal, see e.g., \cite{not11}. If we want to investigate, whether the distribution of the pair correlations for some explicit given sequence is Poissonian, i.e., satisfies property (\ref{eq:pc1}), the sequence has to be uniformly distributed modulo $1$. Therefore, if we investigate the distribution of the spacings between the sequence elements of $(\lbrace 2^n \alpha \rbrace)_{n \in \NN}$, the only reasonable choice for $\alpha$ is a normal number. We obtain the following result. 
\begin{theorem}
The sequence $(\lbrace 2^n \alpha \rbrace)_{n \in \NN}$ where $\alpha$ is the Champernowne constant in base $2$, i.e., $\alpha = 0.1101110010111011 \ldots_2$ \textbf{does not} have Poissonian pair correlations.  
\end{theorem}
This paper was initiated by the conjecture of G.\ Larcher (mentioned during a personal discussion) that all normal numbers are Poissonian, due to the lacunarity of $(\lbrace 2^n \rbrace)_{n \in \NN}$. To make it more tangible why this conjecture is reasonable, we recall that Kronecker sequences are not Poissonian for any $\alpha$ and $(\lbrace \alpha n^d\rbrace)_{n \in \NN}$, $d \geq 2$ is Poissonian for almost all $\alpha$, whereby it is known that $\alpha$ has to satisfy some Diophantine condition, see e.g., \cite{not3}. Hence, one would expect the sequence $(\lbrace 2^n \alpha \rbrace)_{n \in \NN}$ to have the Poissonian property for all normal numbers $\alpha$, as it shows less structure than the Kronecker and polynomial sequences. The motivation to study the sequence described in Theorem 1 was to find the first explicit example of a sequence having Poissonian pair correlations. At least our result allows to immediately deduce that the sequence $(\lbrace 2^n \alpha \rbrace)_{n \in \NN}$ \textbf{cannot} have Poissonian pair correlations \textbf{for all} normal numbers $\alpha$. \\

To prove Theorem 1 we use elementary combinatorics. We give a short outline of the proof. Let $e, d$, be two integers, where $d=2^e$ is understood to be very large compared to $e$. Further, we set $s=1$ and $N=2^{d+e}$ in (\ref{eq:pc1}). The reason for choosing $N$ in such a manner is the following. The Champernowne constant is the concatenation of the numbers which have a digit expansion (with a leading 1) in base $2$ of length $1, 2, 3, \ldots, d, \ldots$ and so forth. In order to account for all blocks of words of length $1, \ldots, d$, we have to choose $N$ large enough, i.e., in our case at the beginning of the block of words having length $d+1$. Note that the length of the block containing the words of length $d$ is $d2^{d-1}=2^{d+e-1}$ and for a very large $e$ all previous blocks of words with length $1, \ldots, d-1$ have in total approximately this length. We then count the occurrence of bit patterns (in the block of words of length $d$), which correspond to shifts of $\alpha$ having a distance (i.e., the Euclidean distance on $\mathbb{R}$) $<1/N$ (we will henceforth abbreviate to simply saying that the patterns have this distance). If we have two patterns which match in the first $d+e$ bits or which are of the form $\underbrace{a_1 a_2 \ldots 0 1 \ldots 1}_{d+e} b_1 b_2 b_3 \ldots$ and $\underbrace{a_1 a_2 \ldots 1 0 \ldots 0}_{d+e} c_1 c_2 c_3 \ldots $ with $c_1 c_2 c_3 \ldots < b_1 b_2 b_3 \ldots$, then their distance is $< 1/N$. It turns out that already the number of pairs matching in the first $d+e$ bits (in the block of words of length $d$) yields a too large contribution. The second case is studied in the appendix. Though the number of pairs (with distance $<1/N$) is small compared to the first case, it is of interest in its own right to see how to count the occurrence of such patterns.   
\section{Proof of the main Theorem}
\begin{proof} 

Let $s=1$ and set $N=2^{d+e}$ where $d$ and $e$ are defined as in the previous section (at first we will not use the relation between $d$ and $e$, though). 
Let a bit pattern $a_1\dots a_w$ be given where $w=d+e$, $e>0$. We are aiming to count
the occurrences of the pattern in the full block 
$$c_{0,1}\dots c_{0,d}\dots c_{2^{d-1}-1,1}\dots c_{2^{d-1}-1,d},$$
and put $c_{n}:=2^{d-1}+n = \sum_{i=1}^d c_{n,i}2^{i-1}$. Note that $c_{i,1}=1$ for $i=0, \ldots, 2^{d-1}-1$.
That is, the pattern has an overlap of $e$ bits to the word length $d$.
The overlap $e$ is understood to be small.

First, we investigate the patterns where the first $e$ bits
match the last ones, i.e., $a_i=a_{d+i}$ for $i=1,\dots,e$. 
We denote the index before the start of a possible matching word by
$z\geq0$, i.e., if a match occurs then there is an $n$ such that
$$a_{z+1}a_{z+2}\dots=c_{n,1}c_{n,2}\dots$$ and at least one of
$$a_{z  }a_{z-1}\dots=c_{n-1,d}c_{n-1,d-1}\dots,$$
$$a_{z+d+1}a_{z+d+2}\dots=c_{n+1,1}c_{n+1,2}\dots.$$

\textbf {Basic Fact (BF1):} for a match to occur, $a_{z}$ must not
equal $a_{z+d}$ since these bits correspond to the least significant
bits of consecutive digit expansions $c_{n-1},c_n$.

\textbf{BF2:}
As a first consequence of BF1, $z$ must be zero or greater
than $e$ since otherwise $a_z=a_{z+d}$ and similarly $z$ must be at
most $d$, else $a_{z-d}=a_{z}$, i.e., $z\in\{0,e+1,\dots,d\}$. 

\textbf{BF3:}
Furthermore, for a match with $z>0$ to occur it is necessary that
$a_{z+1}=1$ and at least one zero occurs in the sequence $a_{e+1}\dots
a_z$. This excludes subpatterns of the forms \[ a_{e+1}\dots a_{z+1} =
1\dots10  \text{ or } 1\dots11 ,  \] which cannot occur due to the
fact that in this case $c_n=c_{n-1}+1$ has carries affecting $a_d\dots
a_{d+e}$.

We now make case distinctions according to the number $k$ of ones in
the `middle block' $a_{e+1}\dots a_d$.

If $k=0$ the pattern can occur in the full block only if $z=0$ and
$a_{z+1}=c_{n,1}=1$ which cannot happen in the middle block or if 
$a_1=0$.

If $k=1$ this type of pattern (or `meta-pattern') can occur 
in the case $a_{e+1}=1$ only if $a_1=1$ and $z=0$; BF3 forbids
$z>0$ and for $z=0$ again $a_{z+1}=a_1=1$ is necessary. If the $1$
appears later in the middle block, again $z=0$ is possible, if $a_1=1$
or $z=j$ if $j+1$ is the index of the $1$. This gives
\begin{itemize}
 \item $2^{e-1} + 2^{e-1}(d-e-1)$ patterns occuring only one time 
 \item $2^{e-1}(d-e-1)$ patterns occuring two times.
\end{itemize}
 
Let us also look at the case $k=2$: 
first, $a_{e+1}=a_{e+2}=1$ by BF2 again necessitates $z=0$ and $a_1=1$,
and can occur only in one match. If $a_{e+1}=1\neq a_{e+2}$ there are
one or two possible matches in dependence of $a_1=0$ or $1$. Finally, 
two or three possible matches can happen if both ones occur later in
the middle block. The tally thus is:
\begin{itemize}
 \item one match: $2^{e-1}(1+(d-e-2))$ patterns
 \item two matches: $2^{e-1}((d-e-2)+\binom{d-e-1}{2})$ patterns
 \item three matches: $2^{e-1}\binom{d-e-1}{2}$ patterns
\end{itemize}

We can now present the general case $2<k<d-e$: 

\begin{itemize}
 \item $a_1=0, a_{e+1}=1,a_{e+2}=0$: 
   we have $2^{e-1} \binom{d-e-2}{k-1} $
  patterns having $k-1$ matches
 \item $a_1=0, a_{e+1}=a_{e+2}=1$: 
 let $a_{e+1}=\dots=a_{e+j}=1\neq a_{e+j+1}$, i.e., there are $j$
 consecutive ones at the start of the middle block followed by a zero.
 Then there are $k-j$ ones left to distribute on $d-e-j-1$ places.
 We have a match for each of those ones, so
 there are\\ $2^{e-1} \binom{d-e-j-1}{k-j}$ patterns having $k-j$
 matches, where $j=1,\dots,k$.  
 \item $a_1=0, a_{e+1}=0$: we have $2^{e-1} \binom{d-e-1}{k} $
  patterns having $k$ matches
 \item $a_1=1, a_{e+1}=0$: we have $2^{e-1}\binom{d-e-1}{k}$
  patterns having $k+1$ matches
 \item $a_1=1, a_{e+1}=1$: 
 let $a_{e+1}=\dots=a_{e+j}=1\neq a_{e+j+1}$, i.e., there are $j$
 consecutive ones at the start of the middle block followed by a zero.
 Then there are $k-j$ ones left to distribute on $d-e-j-1$ places.
 We have a match for each of those ones plus one attributed to $z=0$, i.e.,
 there are\\ $2^{e-1} \binom{d-e-j-1}{k-j}$ patterns having $k-j+1$
 matches, where $j=1,\dots,k$.  
\end{itemize}

The case $k=d-e$ has only patterns matching just once.

Taking all together we get the following formula for the number of
pairs $c_{n_1,i_1},c_{n_2,i_2}$ such that there is a match in 
(at least) $w$ bits. Note that the pairs are ordered. 
\begin{align}
 2^{e} \sum_{k=1}^{d-e-1} 
 \Bigg(&
  \sum_{j=0}^{k-1} \binom{k-j}{2} \binom{d-e-j-1}{k-j} + \nonumber \\ 
   &
  \sum_{j=0}^{k-1} \binom{k-j+1}{2} \binom{d-e-j-1}{k-j} 
 \Bigg) \nonumber
\\=
 2^{e} \sum_{k=1}^{d-e-1} 
 &
  \sum_{j=0}^{k-1} {(k-j)}^{2} \binom{d-e-j-1}{k-j} \nonumber 
\end{align}
By Stirling's formula, $\binom{2n}{n} \gg \frac{2^{2n}}{\sqrt{n}}$, and considering $j=0$ and $k=\lfloor \frac{d-e-1}{2} \rfloor$ in the above sum, we obtain a contribution of magnitude 
\begin{equation*}
2^ek^2 \binom{d-e-1}{k} \gg \frac{d^2}{\sqrt{d}}2^d = \sqrt{d} 2^{d+e}.
\end{equation*}
Therefore, if we divide this amount of pairs by $N = 2^{d+e}$ (recall that we set $d=2^e$) and consider $e \to \infty$, we obtain $+\infty$ in the limit and deduce that the pair correlations distribution cannot be asymptotically Poissonian.  \\

For the sake of completeness, we study two further types of patterns. We will see that these two structures of patterns yield a negligible amount of pairs.   
The next type of pattern is where the matching $c_n$ ends in a string
of ones, inducing a chain of carries for $c_{n+1}$. I.e., there are
$j_0,j_1$, $1\leq j_0 \leq e < j_1\leq d-1$ such that 
\[
 a_{j_0}a_{j_0+1}\dots a_e a_{e+1}\dots a_{j_1}a_{j_1+1} =
 01\dots11\dots10
\]
and $a_i=a_{d+i}$ for $1\leq i<j_0$, $a_i=1-a_{d+i}$ for $j_0\leq
i\leq e$.  Again, a possible matching $c_n$ can obviously not start
with an index earlier than $e+1$ since then inevitably mismatches
$a_i\neq a_{d+i}$ that cannot be accounted for by carries occur. But
then each of the consecutive ones can be taken as start of a
$c_n$-block, i.e., $z=e,\dots,j_1-1$ are all possible, giving $j_1-e$
matches.  Given $j_0,j_1$ there are $2^{j_0-1+\,d-j_1-1}$ such patterns.
For the case $j_1=d$ there are $d-e$ matches as well, plus an
additional one if $a_1=1$. Both subcases have $2^{j_0-2}$ according
patterns for $j_0\geq2$ and additionally there is one further case
for $j_0=1$ with $d-e$ matches, the pattern $01^{d}0^{e-1}$.
The number of ordered pairs thus equals:
\begin{align*}
 2 \Bigg( &\sum_{j_0=1}^{e} \sum_{j_1=e+1}^{d-1} 
       \binom{j_1-e}{2} 2^{j_0+d-j_1-2} \\&
 + \left( \binom{d-e}{2}+\binom{d-e+1}2\right) 
  \sum_{j_0=2}^e 2^{j_0-2}
 + \binom{d-e}2 \Bigg)  \\
= \frac{2^e-1}{2^{e-1}}&(2^d-2^e)-(d-e)2^e   
\quad <\quad 2^{d+1}.
\end{align*}

The next type of pattern, which only yields a negligible amount of relevant pairs, is the one where $a_1 \ldots a_e =1  \ldots 1$ and $a_{e+1} \ldots a_{z+1} = 1  \ldots 1$, where $z  \in \lbrace e, \ldots, d-1 \rbrace$. As a consequence thereof $a_{d+1} \ldots a_{d+e} = 0 \ldots 0$. Hence, we have 
\begin{equation*}
\sum_{i=3}^{d-e} \binom{i-1}{2} = \frac{1}{6} (d-e-2)(d-e-1)(d-e)
\end{equation*}
pairs with distance $ < 1/N$. 
\end{proof}
\begin{remark}
In the proof we have only studied the case, where a fixed bit pattern of length $w$ overlaps two words of length $d$. Of course, an overlap of the pattern with three words might also occur, but these cases yield a small number of pairs with prescribed distance. Therefore we have omitted the exact study of these structures. If the relative number of pairs in the block of words of length $d$ would have given a number less than $2s$, then a study of the occurrence of the pattern in the block of words of length $d-1$ (and so forth) would have been necessary.    
\end{remark}
\begin{remark}
The techniques from above can of course be adapted to any other base $b$, i.e., we can conclude that the Champernowne constant in base $b \geq 2$ (note that the Champernowne constant in base $b$ is normal to base $b$) is not Poissonian.
\end{remark}
\section{Open Problems and Outlook}
In this section we first want to state an open problem, which involves the notion of weak pair correlations (introduced by Steinerberger in \cite{not12}), a concept that relaxes the requirements of (\ref{eq:pc1}). \\
  
We state the following open problem. \\

\textbf{Problem 1:} Does the sequence $(x_n)_{n \in \NN}= (\lbrace 2^n \alpha \rbrace)_{n \in \NN}$, where $\alpha$ is the Champernowne constant in base $2$, satisfy the notion of weak pair correlation, i.e., is there an $0 <\beta < 1$, such that 
\begin{equation*}
\lim_{N \to \infty} \frac{1}{N^{2- \beta}} \# \left\lbrace 1 \leq l \neq m \leq N: \| x_l - x_m \| \leq \frac{s}{N^{\beta}}  \right\rbrace = 2s
\end{equation*} 
for every $s \geq 0$? \\

Further, we still need to find an explicit construction of an $\alpha$ such that a sequence of the form $(\lbrace a_n \alpha \rbrace)_{n \in \NN}$ has Poissonian pair correlations and maybe criteria which relax the definition of Poissonian pair correlations, e.g., that it possibly suffices to show that (\ref{eq:pc1}) holds for $s \in \mathbb{N}$ only. A possible approach would be to modify the Champernowne constant in a certain way, e.g., by shifts, such that we avoid the situation that we have too many patterns where the first and last $e$ bits match.     

\section{Appendix}
Though the here presented results are not needed for the proof of Theorem 1, they give additional interesting information about the pair correlation structure of the Champernowne constant and therefore we add them as appendix. \\

In the previous section we have counted the occurrence of a bit pattern $a_1 \ldots a_w$ in the full block of words of length $d$. Now, we consider patterns of the form  $b:=a_1 a_2 \ldots a_w b_1b_2b_3 \ldots = \underbrace{a_1 a_2 \ldots a_j 0 1 \ldots 1}_{w} b_1 b_2 b_3 \ldots$ and $c:=a^{'}_1 a^{'}_2 \ldots a^{'}_w c_1c_2c_3 \ldots = \underbrace{a_1 a_2 \ldots a_j 1 0 \ldots 0}_{w} c_1 c_2 c_3 \ldots$, with $b_1 b_2 b_3 \ldots > c_1 c_2 c_3 \ldots$. These two types of bit words also have a distance less than $1/N$. \\

We only study the case $b_1=1$ and $c_1=0$ in detail. The ideas used for the special case can be readily generalized. Therefore, we are aiming at counting the occurrences of bit blocks of the form $B: = \underbrace{a_1 a_2 \ldots a_j 0 1 \ldots 1}_{w} 1$ (in the full block of words of length $d$) and the ones of the form $C: =\underbrace{a_1 a_2 \ldots a_j 1 0 \ldots 0}_{w} 0$. \\
\begin{coro}
The patterns of the form $B$ and $C$ yield for $j=d$ at least 
\begin{equation}\label{eq:eq122}
2^{d-e-1}(d-e-5)
\end{equation}
pairs with distance less than $1/N$. For $j > d$ we obtain at least 
\begin{equation}\label{eq:eq123}
2^{-1 - e} (2^e-2) (2^{2 + e} + 2^d d + 2^{1 + e} d - 
   2^d e - 2^{1 + e} e - 2^{2 + d})
\end{equation}
pairs with distance less than $1/N$. 
\end{coro}
\begin{proof}
We start by studying the occurrence of the first pattern. Note that we first consider the case, where the first and the last $e$ bits of $\underbrace{a_1 a_2 \ldots a_j 0 1 \ldots 1}_{w}$ match. In the following, we distinguish several cases, depending on the position of the index $j$. Later, we will see that the only relevant cases are the ones where $j \geq d$. Thus, we will examine only those in more detail. 
\begin{itemize}
\item $j=d$: 
\begin{itemize}
\item First, let $a_{e+1} =1$ and due to $j=d$, $a_1a_2 \ldots a_e = a_{d+1}a_{d+2} \ldots a_{d+e} = 0 1 \ldots 1$. Let $k$ be the number of ones in the block $a_{e+1} \ldots a_{d}$. Then, we obtain
\begin{equation*}
\sum_{k=2}^{d-e-1} \sum_{l=1}^{k-1}(k-l) \binom{d-e-l-1}{k-l}
\end{equation*}
matches. 
\item Consider now $a_{e+1}=0$. If there exists $z \leq d$ with $a_{e+1} a_{e+2} \ldots a_{z} = 01 \ldots 1$, then this case yields 
\begin{equation*}
\sum_{k=1}^{d-e-1} \sum_{l=1}^{k} l \binom{d-e-l-2}{k-l}
\end{equation*} 
matches.  
\end{itemize}
\item $j>d$: 
\begin{itemize}
\item Let $a_{e+1}=1$. We have the structure (the first and last $e$ digits are again equal)  $a_1 \ldots a_e = a_{d+1} \ldots a_j a_{j+1} \ldots a_{d+e} = a_{d+1} \ldots a_{j} 0 \ldots 1$. In total there are 
\begin{equation*}
2^{j-d-1}\Bigg[ \sum_{k=2}^{d-e-1} \Bigg( \sum_{l=1}^{k}  (k-l) \binom{d-e-l-1}{k-l} + (k-l+1)\binom{d-e-l-1}{k-l} \Bigg) \Bigg]
\end{equation*}  
matches.
\item Let $a_{e+1}=0$. Here, we get (similar to above)
\begin{equation*}
2^{j-d} \Bigg( \sum_{k=1}^{d-e-1} \sum_{l=1}^{k} l \binom{d-e-l-2}{k-l} \Bigg)
\end{equation*}  
matches.
\end{itemize} 
\end{itemize}  
A similar study can be carried out for the structure of the second pattern mentioned at the beginning. \\ \\
It remains to check if above cases allow starting and end blocks of the form $a_1 \ldots a_e = a_1 \ldots a_{m-1} 0 1 \ldots 1$ and $a_{d+1} \ldots a_{d+e} = a_1 \ldots a_{m-1} 1 0 \ldots 0$, respectively. If $d > j > e$, or $ j  \leq e$, then this cannot happen. The case $j \geq d$ allows starting and end blocks of this form for the second pattern. \\ \\
Above, we have investigated how often (and for which cases) one of the two patterns $B$ and $C$ occurs. It remains to analyse how many matches of the respective patterns agree in the first $j$ digits. The only relevant cases are where $j \geq d$. 
\begin{itemize} 
\item $j=d$: Here we have for the pattern $B$ the structure \\ $a_1a_2 \ldots a_e = a_{d+1}a_{d+2} \ldots a_{d+e} = 0 1 \ldots 1$. For the pattern $C$ we have the feasible structure $a_1a_2 \ldots a_e = 0 1 \ldots 1$ and $a_{d+1}a_{d+2} \ldots a_{d+e} = 1 0 \ldots 0$. I.e., we obtain
\begin{equation*}
2 \sum_{k=2}^{d-e-1} \sum_{l=1}^{k-1} (l-1)(k-l) \binom{d-e-l-1}{k-l} 
\end{equation*}
pairs with distance $< 1/N$. Note that the last equation can be simplified to (\ref{eq:eq122}). 
\item $j > d$: Here, we therefore get
\begin{align*}
2^{j-d}\Bigg[ &\sum_{k=2}^{d-e-1} \Bigg( \sum_{l=1}^{k}  (l-1)(k-l)  \binom{d-e-l-1}{k-l} \\
&+ (l-1)(k-l+1)\binom{d-e-l-1}{k-l} \Bigg) \Bigg] \nonumber
\end{align*}
pairs with distance $< 1/N$. Summation for $d+1 \leq  j \leq d+e-1$ yields (\ref{eq:eq123}). \\
To make this formula more tangible, we note that for the first pattern we have shifts according to the number of ones after the first zero in the middle block $a_{e+1} \ldots a_{d}$. If $a_{d+1}=1$, then we have one additional shift. The second pattern allows (due to the necessary carry and $c_1=0$) $l-1$ shifts, where $l$ denotes the number of ones at the beginning of the block $a_{e+1} \ldots a_{d}$.  
\end{itemize} 
\end{proof}
\noindent 
\textbf{Acknowledgements.} We would like to thank Gerhard Larcher for his valuable comments, suggestions and his general assistance during the writing of this paper. 

\textbf{Author’s Addresses:} \\ 
\'Isabel Pirsic  and Wolfgang Stockinger, Institut f\"ur Finanzmathematik und Angewandte Zahlentheorie, Johannes Kepler Universit\"at Linz, Altenbergerstraße 69, A-4040 Linz, Austria. \\ \\ 
Email: isabel.pirsic(at)jku.at, wolfgang.stockinger(at)jku.at 
\end{document}